\documentclass[english]{article}
\usepackage[T1]{fontenc}
\usepackage[latin9]{inputenc}
\usepackage{geometry}
\usepackage{enumitem}
\usepackage{amsmath}
\usepackage{amsthm}
\usepackage{amssymb}
\usepackage{enumitem}
\usepackage{mathtools}
\usepackage[textsize=scriptsize,colorinlistoftodos]{todonotes}

\makeatletter
\numberwithin{equation}{section}
\numberwithin{figure}{section}
\theoremstyle{plain}
\newtheorem{theorem}{\protect\theoremname}
  \theoremstyle{plain}
  \newtheorem{lemma}[theorem]{\protect\lemmaname}
  \theoremstyle{plain}
  
  \theoremstyle{plain}
  \newtheorem{corollary}[theorem]{\protect\corollaryname}
  \theoremstyle{plain}
  \newtheorem{conjecture}[theorem]{\protect\conjecturename}
  \theoremstyle{definition}
  \newtheorem{definition}[theorem]{\protect\definitionname}
 \newlist{casenv}{enumerate}{4}
 \setlist[casenv]{leftmargin=*,align=left,widest={iiii}}
 \setlist[casenv,1]{label={{\itshape\ \casename} \arabic*.},ref=\arabic*}
 \setlist[casenv,2]{label={{\itshape\ \casename} \roman*.},ref=\roman*}
 \setlist[casenv,3]{label={{\itshape\ \casename\ \alph*.}},ref=\alph*}
 \setlist[casenv,4]{label={{\itshape\ \casename} \arabic*.},ref=\arabic*}

\makeatother

\usepackage{babel}
  \providecommand{\corollaryname}{Corollary}
  \providecommand{\definitionname}{Definition}
  \providecommand{\lemmaname}{Lemma}
  \providecommand{\propositionname}{Proposition}
  \providecommand{\conjecturename}{Conjecture}
  \providecommand{\casename}{Case}
\providecommand{\theoremname}{Theorem}
\title{Linear dependence between hereditary quasirandomness conditions}

\author{Xiaoyu He}

\begin{document}

\maketitle

\begin{abstract}
Answering a question of Simonovits and S\' os, Conlon, Fox, and Sudakov
proved that for any nonempty graph $H$, and any $\varepsilon>0$, there
exists $\delta>0$ polynomial in $\varepsilon$, such that if $G$
is an $n$-vertex graph with the property that every $U\subseteq V(G)$
contains $p^{e(H)}|U|^{v(H)}\pm\delta n^{v(H)}$ labeled copies of
$H$, then $G$ is $(p,\varepsilon)$-quasirandom in the sense that
every subset $U\subseteq G$ contains $\frac{1}{2}p|U|^{2}\pm\varepsilon n^{2}$
edges. They conjectured that $\delta$ may be taken to be linear in
$\varepsilon$ and proved this in the case that $H$ is a complete graph.
We study a labelled version of this quasirandomness property proposed
by Reiher and Schacht. Let $H$ be any nonempty graph on $r$ vertices
$v_{1},\ldots,v_{r}$, and $\varepsilon>0$. We show that there exists
$\delta=\delta(\varepsilon)>0$ linear in $\varepsilon$, such that
if $G$ is an $n$-vertex graph with the property that every sequence
of $r$ subsets $U_{1},\ldots,U_{r}\subseteq V(G)$, the number of
copies of $H$ with each $v_{i}$ in $U_{i}$ is $p^{e(H)}\prod|U_{i}|\pm\delta n^{v(H)}$,
then $G$ is $(p,\varepsilon)$-quasirandom. 
\end{abstract}

\section{Introduction}

Random-like objects, in particular quasirandom graphs, have become
a central object of study in combinatorics and theoretical computer
science (see for example the survey article of Krivelevich and Sudakov
~\cite{KrivelevichSudakov}). In this paper, we will prove a generalization
of a result of Conlon, Fox, and Sudakov~\cite{ConlonFoxSudakov} on
quasirandom graphs that ties in with a line of research motivated
by two important principles of extremal graph theory. First, that
many ``natural'' properties of random graphs are equivalent; second,
that many results provable by Szemer\' edi's regularity lemma can
be more effectively proved directly without it.

Although certain notions of quasirandom graphs were studied earlier,
such as in Thomason's work on jumbled graphs~\cite{Thomason1},
the idea that many of these notions are equivalent first appeared
in the seminal work of Chung, Graham and Wilson~\cite{ChungGrahamWilson}. 

The Erd\H os-R\'enyi random graph $G(n,p)$ is the random graph
on $n$ vertices where each of the $\binom{n}{2}$ edges is drawn
independently with probability $p$. A priori, any number of properties
of the prototypical random graph $G(n,p)$ could be used to define
quasirandomness, but Chung, Graham, and Wilson~\cite{ChungGrahamWilson}
discovered that many of these properties are qualitatively equivalent,
leading to a canonical notion of quasirandomness for graphs. 

We will only consider simple, undirected graphs. Write $V(G)$ for the set of vertices of a graph $G$, $E(G)$ for
the set of edges, and define $v(G)=|V(G)|$, $e(G)=|E(G)|$. Also,
write $x=y\pm\Delta$ if $|x-y|\le\Delta$. We say that a graph $G$
has edge density $p$ if $e(G) = p\binom{v(G)}{2}$.
\begin{theorem}
(Chung, Graham, and Wilson~\cite{ChungGrahamWilson}). Let $p\in[0,1]$.
The following are equivalent properties of a graph $G$ with edge
density $p$, up to the choice of $\varepsilon>0$:
\begin{enumerate}
\item For some $s\ge4$ and every graph $H$ on $s$ vertices, the number
of induced subgraphs of $G$ isomorphic to $H$ is
\[
p^{e(H)}(1-p)^{\binom{v(H)}{2}-e(H)}v(G)^{v(H)}\pm\varepsilon v(G)^{v(H)}.
\]
\item There exists a nontrivial graph $H$ such that every induced subgraph
$G[U]$ of $G$ contains $p^{e(H)}|U|^{v(H)}\pm\varepsilon v(G)^{v(H)}$
(not necessarily induced) subgraphs isomorphic to $H$.
\item The number of $4$-cycles in $G$ is at most $p^{4}v(G)^{4}+\varepsilon v(G)^{4}$.
\item The two largest (in absolute value) eigenvalues $\lambda_{1},\lambda_{2}$
of the adjacency matrix of $G$ satisfy $\lambda_{1}=(p\pm\varepsilon)v(G)$
and $|\lambda_{2}|\le\varepsilon v(G)$.
\item For every vertex subset $U$ of $G$, the number of edges in $U$
satisfies $e(U)=\frac{1}{2}p|U|^{2}\pm\varepsilon v(G)^{2}$.
\end{enumerate}
\end{theorem}
Explicitly, we mean that for any properties $i\ne j\in\{1,2,3,4,5\}$
above, and any $\varepsilon>0$, there exists $\varepsilon'>0$ such
that any graph $G$ with density $p$ which satisfies property $i$
with error parameter $\varepsilon'$ also satisfies property $j$
with error parameter $\varepsilon$.

Although these conditions are equivalent, the relationships between
the various $\varepsilon$'s are not so well understood. For example,
let $\mathcal{P}_{H,p}^{*}(\varepsilon)$ be property $2$ above for
some fixed $H,p$ and $\varepsilon>0$, which is called the ``hereditary
quasirandomness'' condition because it is inherited by all induced
subgraphs. Simonovits and S\'os~\cite{SimonovitsSos} were able
to prove using Szemer\' edi's regularity lemma~\cite{Szemeredi}
that for any two graphs $H,H'$, $\mathcal{P}_{H,p}^{*}(\delta)\implies\mathcal{P}_{H',p}^{*}(\varepsilon)$
where $\delta^{-1}$ is growing as a tower function of $\varepsilon^{-1}$.
The Simonovits-S\'os conjecture is that this dependence can be proved
without the regularity lemma.

Roughly speaking, Szemer\' edi's regularity lemma states that given
any $\varepsilon>0$, every graph $G$ can be decomposed into $K=K(\varepsilon)$
parts $V_{1},\ldots,V_{K}$ (the ``regularity partition'') such
that the edges between most pairs $V_{i},V_{j}$ are within $\varepsilon$
of being random. Although the regularity lemma is an extraodinarily
powerful tool, the quantative dependence of $K$ on $\varepsilon$
is of tower-type growth and this cannot be improved. Thus, tower-type
quantitative dependency is indicative of a straightforward application
of the regularity lemma. 

In practice, the regularity lemma is not always necessary, and we
can expect to improve quantitative bounds in various applications
by avoiding its use. When this is possible, the tower-type bounds
provided by the regularity lemma can usually be replaced by exponential
or even polynomial bounds. One of the most important examples of this
method is the weak regularity lemma of Frieze-Kannan~\cite{FriezeKannan},
which proves that a graph $G$ can be decomposed into a regularity
partition with only exponentially many parts in $\varepsilon$ if
we replace the regularity condition by a weaker global version. Problems
amenable to the regularity method sometimes only require the Frieze-Kannan
weak regularity condition. For example, it is possible to prove the
Simonovits-S\'os Conjecture with exponential dependence of $\delta$
on $\varepsilon$ using this method.

The Simonovits-S\'os Conjecture was settled by Conlon, Fox, and Sudakov
~\cite{ConlonFoxSudakov} by carefully extracting the useful ingredients
of the regularity proof without using its full power, achieving polynomial
dependency of $\delta^{-1}$ on $\varepsilon^{-1}$. They further
conjectured that the true dependence is linear. In this paper we will
study a variation of this conjecture introduced by Reiher and Schacht
~\cite{ReiherSchacht}, where instead of counting copies of an unlabelled
graph $H$ in any induced subgraph $G[U]$ of $G$, we count labelled
homomorphic copies of $H$ with the $i$-th vertex $v_{i}$ lying
in a prescribed subset $U_{i}\subset V(G)$. In this situation, we
prove the optimal linear dependence by extending a counting argument
from~\cite{ConlonFoxSudakov}.

\section{Background}

We first offer some notation for counting subgraphs. If $G$ is a graph and $U\subseteq V(G)$, write
$G[U]$ for the induced subgraph of $G$ on $U$. 

Let $H,G$ be two labelled graphs where $H$ has $r$ vertices $v_{1},\ldots,v_{r}$.
Let $U_{1},U_{2},\ldots,U_{r}$ be vertex subsets of $G$. In this setting, we define
$c(H,G;U_{1},\ldots,U_{r})$ to be the number of (labelled graph)
homomorphisms $\phi:H\rightarrow G$ with $\phi(v_{i})\in U_{i}$.
We abbreviate $c(H,G)$ for the total number of homomorphisms $H\rightarrow G$.
In particular $c(H,G[U])=c(H,G;U,\ldots,U)$ is the number of homomorphisms
from $H$ to the induced subgraph of $G$ on $U$. We think of $c(H,G;U_{1},\ldots,U_{r})$
as the number of (non-induced) labelled copies of $H$ in $G$ with
each vertex in a predetermined subset.
\begin{definition}
If $H$ is a fixed graph, $p\in[0,1]$ and $\varepsilon>0$, we say
$G$ satisfies the hereditary quasirandomness property $\mathcal{P}_{H,p}^{*}(\varepsilon)$
if for all $U\subseteq V(G)$,
\[
c(H,G[U])=p^{e(H)}|U|^{v(H)}\pm\varepsilon v(G)^{v(H)}.
\]
\end{definition}
In other words, the condition $\mathcal{P}_{H,p}^{*}(\varepsilon)$
is that every induced subgraph contains the right number of copies
of $H$. Note for any fixed $\varepsilon>0$ and $n$ sufficiently
large, the random graph $G(n,p)$ satisfies $\mathcal{P}_{H,p}^{*}(\varepsilon)$
almost surely. Also, note that $\mathcal{P}_{H,p}^{*}(\varepsilon)$ is trivially
satisfied by every graph if $H$ is empty, so we will only be concerned with nonempty $H$.

Simonovits and S\'os~\cite{SimonovitsSos} proved using the Szemer\'edi's
regularity lemma that the properties $\mathcal{P}_{H,p}^{*}(\varepsilon)$
are all equivalent, in the sense that for any nontrivial graphs $H,H'$
and $\varepsilon>0$, there exists $\delta>0$ such that $\mathcal{P}_{H,p}^{*}(\delta)\implies\mathcal{P}_{H',p}^{*}(\varepsilon)$.
Unfortunately, the dependence of $\delta$ on $\varepsilon$ in their
proof is of tower type because of the use of the regularity lemma.
Note that it suffices to show that $\mathcal{P}_{H,p}^{*}(\delta)\implies\mathcal{P}_{K_{2},p}^{*}(\varepsilon)$
where $K_{2}$ is the graph with a single edge; the other direction
is given by a straightforward counting lemma, which we state as Lemma~\ref{lem:counting} below.

Conlon, Fox, and Sudakov~\cite{ConlonFoxSudakov} were able to tailor
the regularity method to this problem to prove the same result with
polynomial dependence of the form $\delta=\Omega(\varepsilon^{f(p,v(H))})$
where the exponent $f$ depends only on $p$ and $v(H)$. They conjectured
that the dependence is in fact linear, and proved it for the case
$H=K_{n}$.
\begin{conjecture}
\label{conj:SSlinear}For any nonempty graph $H$, and real numbers
$p\in[0,1]$, $\delta>0$, we have
\[
\mathcal{P}_{H,p}^{*}(\delta)\implies\mathcal{P}_{K_{2},p}^{*}(\varepsilon)
\]
for some $\varepsilon=O_{H,p}(\delta)$.
\end{conjecture}
Independently, Reiher and Schacht~\cite{ReiherSchacht} showed a similar
polynomial dependence for a stronger notion of quasirandomness which
takes configurations into account.
\begin{definition}
If $H$ is a fixed graph, we say $G$ satisfies $\mathcal{R}_{H,p}(\varepsilon)$
if for every sequence of $v(H)$ disjoint vertex subsets $U_{1},\ldots,U_{v(H)}\subseteq G$,
\[
c(H,G;U_{1},\ldots,U_{v(H)})=p^{e(H)}\prod_{i=1}^{v(H)}|U_{i}|\pm\varepsilon v(G)^{v(H)}.
\]
\end{definition}
In this paper we show the linear dependence in Conjecture~\ref{conj:SSlinear}
using the stronger condition $\mathcal{R}_{H,p}(\varepsilon)$. Note that as with $\mathcal{P}_{H,p}^{*}(\delta)$, the condition $\mathcal{R}_{H,p}(\varepsilon)$ is trivial when $H$ is empty. 
\begin{theorem}
\label{thm:main}For any nonempty graph $H$, and real numbers $p\in[0,1]$,
$\delta>0$, we have
\[
\mathcal{R}_{H,p}(\delta)\implies\mathcal{P}_{K_{2},p}^{*}(\varepsilon)
\]
for some $\varepsilon=O_{H}(p^{-3e(H)}\delta)$.
\end{theorem}
The converse is a standard counting lemma; we show it in Lemma~\ref{lem:counting}.

We will begin by proving that $\mathcal{R}_{H,p}(\varepsilon)$ is
equivalent up to linear change of $\varepsilon$ to $\mathcal{R}'_{H,p}(\varepsilon)$,
which is the same condition with disjointness removed.
\begin{definition}
If $H$ is a fixed graph, we say $G$ satisfies $\mathcal{R}'_{H,p}(\varepsilon)$
if for every sequence of $v(H)$ (not necessarily disjoint) vertex
subsets $U_{1},\ldots,U_{v(H)}\subseteq G$,
\[
c(H,G;U_{1},\ldots,U_{v(H)})=p^{e(H)}\prod_{i=1}^{v(H)}|U_{i}|\pm\varepsilon v(G)^{v(H)}.
\]
\end{definition}
After this simple argument, we show that the argument of Conlon, Fox,
and Sudakov which proves Conjecture~\ref{conj:SSlinear} for $H=K_{n}$
extends naturally to all $H$ under the stronger condition $\mathcal{R}'_{H,p}(\varepsilon)$.
Note that up to linear change in $\varepsilon,$ $\mathcal{R}'_{H,p}(\varepsilon)$
is equivalent to $\mathcal{P}_{H,p}^{*}(\varepsilon)$ when $H$ is
a complete graph.

The original conjecture of Conlon, Fox, Sudakov remains open. By Theorem~\ref{thm:main}, it would suffice to show that $\mathcal{P}_{H,p}^{*}(\varepsilon)$
and $\mathcal{R}_{H,p}(\varepsilon)$ are equivalent up to linear
change in $\varepsilon$.
\begin{conjecture}
For any graph $H$ and real numbers $p\in[0,1]$, $\delta>0$,
we have
\[
\mathcal{P}_{H,p}^{*}(\delta)\implies\mathcal{R}_{H,p}(\varepsilon)
\]
for some $\varepsilon=O_{H,p}(\delta)$.
\end{conjecture}
The other direction is easy by inclusion-exclusion.

\section{Preliminaries}

Here we reduce $\mathcal{R}_{H,p}(\varepsilon)$ to $\mathcal{R}'_{H,p}(\varepsilon)$
and then survey some standard lemmas that we will need from graph
theory.
\begin{lemma}
\label{lem:disjoint}For any graph $H$ and real numbers $p\in[0,1]$,
$\delta>0$, we have
\[
\mathcal{R}_{H,p}(\delta)\implies\mathcal{R}'_{H,p}(\varepsilon)
\]
for some $\varepsilon=O_{H}(\delta)$.
\end{lemma}
\begin{proof}
If $K$ is a graph on the integers in $[1,v(H)]$, we say a graph
$G$ satisfies condition $\mathcal{R}_{H,p}^{K}(\varepsilon)$ if
for every sequence of vertex subsets $U_{1},\ldots,U_{v(H)}\subseteq G$
such that $U_{i}\cap U_{j}=\emptyset$ whenever $(i,j)$ is an edge
of $K$,
\[
c(H,G;U_{1},\ldots,U_{v(H)})=p^{e(H)}\prod_{i=1}^{v(H)}|U_{i}|\pm\varepsilon v(G)^{v(H)}.
\]

When $K$ is complete, $\mathcal{R}_{H,p}^{K}(\varepsilon)$ is exactly
$\mathcal{R}_{H,p}(\varepsilon)$ and when $K$ is empty it is $\mathcal{R}'_{H,p}(\varepsilon)$.
To prove the lemma inductively, it suffices to show that if $K'$
is a graph containing $K$ with has one more edge, $\mathcal{R}_{H,p}^{K'}(\delta)\implies\mathcal{R}_{H,p}^{K}(\varepsilon)$
for some $\varepsilon=O_{H}(\delta)$. Let $G$ satisfy $\mathcal{R}_{H,p}^{K'}(\delta)$.
Without loss of generality, suppose $K'=K\cup(1,2)$. Let $U_{1},\ldots,U_{v(H)}$
be a sequence of vertex subsets of $G$ such that $U_{i}\cap U_{j}=\emptyset$
whenever $(i,j)$ is an edge of $K$. We will to show that
\[
c(H,G;U_{1},\ldots,U_{v(H)})=p^{e(H)}\prod_{i=1}^{v(H)}|U_{i}|\pm(6+o(1))\delta v(G)^{v(H)},
\]
where the $o(1)$ goes to zero as a function of $v(G)$.

If $U_{1}\cap U_{2}=\emptyset$ we are immediately done, since $\mathcal{R}_{H,p}^{K'}(\varepsilon)$
already applies. Otherwise, write
\begin{eqnarray*}
c(H,G;U_{1},U_{2},\ldots) & = c(H,G;U_{1}\backslash U_{2},U_{2},\ldots) + c(H,G;U_{1}\cap U_{2},U_{2}\backslash U_{1},\ldots) \\
& + c(H,G;U_{1}\cap U_{2},U_{1}\cap U_{2},\ldots).
\end{eqnarray*}

Directly applying $\mathcal{R}_{H,p}^{K'}(\delta)$ to the first two
terms, we obtain
\begin{eqnarray*}
c(H,G;U_{1},U_{2},\ldots) & = & p^{e(H)}\Big(|U_{1}||U_{2}|-|U_{1}\cap U_{2}|^{2}\Big)\prod_{i=3}^{v(H)}|U_{i}|\pm2\delta v(G)^{v(H)}\\
 &  & +c(H,G;U_{1}\cap U_{2},U_{1}\cap U_{2},\ldots).
\end{eqnarray*}

It remains to show that for any $U\subseteq V(G)$,
\[
c(H,G;U,U,U_{3},\ldots,U_{v(H)})=p^{e(H)}|U|^{2}\prod_{i=3}^{v(H)}|U_{i}|\pm(4+o(1))\delta v(G)^{v(H)}.
\]

For this, we may assume (by adding a vertex if necessary) that $|U|$
is even, since a single vertex lies in $O(v(G)^{v(H)-1})$ copies
of $H$. Pick $U_{1}',U_{2}'$ to be a uniform random equitable bipartition
of $U$, i.e. dividing $U$ into two subsets of equal order. The number
of homomorphisms $\phi:H\rightarrow G$ with $\phi(v_{1})=\phi(v_{2})$
is $O(v(G)^{v(H)-1})$. Apart from these, each homomorphism in $c(H,G,U,U,U_{3},\ldots)$
is counted in $c(H,G;U_{1}',U_{2}',U_{3},\ldots)$ with probability
$1/4$, and so by linearity of expectation,
\[
\mathbb{E}[c(H,G;U_{1}',U_{2}',U_{3},\ldots)]=\frac{1}{4}c(H,G;U,U,U_{3},\ldots)+O(v(G)^{v(H)-1}).
\]

On the other hand, because $U_{1}'$ and $U_{2}'$ are disjoint, the
expression inside the expectation is bounded by the property $\mathcal{R}_{H,p}^{K'}(\delta)$.
Thus, reversing the last equation gives
\begin{eqnarray*}
c(H,G;U,U,U_{3},\ldots) & =4\mathbb{E}[c(H,G;U_{1}',U_{2}',U_{3},\ldots)]+O(v(G)^{v(H)-1})\\
 & =p^{e(H)}\prod_{i=1}^{v(H)}|U_{i}|\pm(4+o(1))\delta v(G)^{v(H)},
\end{eqnarray*}
as desired. In particular, since the induction takes $\binom{v(H)}{2}$
steps and at each step the constant factor is at most $6+o(1)$, we
have proved that $\mathcal{R}_{H,p}(\delta)\implies\mathcal{R}'_{H,p}(\varepsilon)$
for some $\varepsilon>0$ satisfying
\[
\varepsilon\le\Big(6^{\binom{v(H)}{2}}+o(1)\Big)\delta.
\]\end{proof}
The next lemma is needed to give a preliminary lower bound on the
edge density of a graph satisfying $\mathcal{R}_{H,p}(\delta)$. It
is a corollary of a stronger result of Alon~\cite{Alon}, who determined
the asymptotic order of the maximum number of copies of $H$ in any
graph $G$ with a prescribed number of edges.
\begin{lemma}
\label{lem:maxcopies}If $H$ is a graph with no isolated vertices,
then for any graph $G$,
\[
c(H,G)=O(e(G)^{v(H)}).
\]
\end{lemma}
We also recall a standard counting lemma, see for example Section
10.5 of Lov\' asz' problem book~\cite{Lovasz}. It tells us how to
count copies of $H$ given quasirandomness. If $A,B\subseteq V(G)$,
let $e(A,B)=c(K_{2},G;A,B)$ be the number of edges between $A$ and
$B$, defined so that if $A$ and $B$ intersect we count each edge
within $G[A\cap B]$ twice. In particular, $e(A,A)=2e(A)$, since
the former counts labelled edges.
\begin{lemma}
\label{lem:counting} If $G$ is a graph which satisfies $\mathcal{P}_{K_{2},p}^{*}(\delta)$
then $G$ satisfies $\mathcal{R}_{H,p}^{'}(4e(H)\delta)$ for all
graphs $H$.
\end{lemma}
\begin{proof}
Suppose $G$ satisfies $\mathcal{P}_{K_{2},p}^{*}(\delta)$. We wish
to show that for any graph $H$ (on vertices $v_{1},\ldots,v_{v(H)}$)
and any sets $U_{1},\ldots,U_{v(H)}\subseteq V(G)$,
\begin{equation}
c(H,G;U_{1},\ldots,U_{v(H)})=p^{e(H)}\prod_{i=1}^{v(H)}|U_{i}|\pm4e(H)\delta v(G)^{v(H)}.\label{eq:counting general}
\end{equation}

Let $c_{H}=c(H,G;U_{1},\ldots,U_{v(H)})$ be the desired homomorphism
count. To prove \eqref{eq:counting general}, we will expand $c_{H}$
as a sum involving the indicator functions of edges, and prove that
it is possible to approximate these indicator functions with the constant
function $p$. Since 
\[
e(A,B)=e(A\cup B)+e(A\cap B)-e(A\backslash B)-e(B\backslash A),
\]
the fact that $G$ satisfies $\mathcal{P}_{K_{2},p}^{*}(\delta)$
implies that for every pair of vertex subsets $A,B\subseteq V(G)$,
\[
\Big|e(A,B)-p|A||B|\Big|\le4\delta v(G)^{2}.
\]

Let $1_{G}(u,v)$ be the indicator function of edges of $G$. Another
way of writing the above inequality is that for any two functions
$f,g:V(G)\rightarrow\{0,1\}$ (which will be the indicator functions
of some two sets $A$ and $B$),
\begin{equation}
\Big|\sum_{u,v\in V(G)}f(u)g(v)(1_{G}(u,v)-p)\Big|\le4\delta v(G)^{2}.\label{eq:regularity1}
\end{equation}

We can expand $c_{H}$ in terms of the indicator function $1_{G}$,
giving
\[
c_{H}=\sum_{(u_{1},\ldots,u_{v(H)})}\prod_{(v_{i},v_{j})\in E(H)}1_{G}(u_{i},u_{j}),
\]
where the sum is over all $v(H)$-tuples of vertices $(u_{i})_{i=1}^{v(H)}$
with $u_{i}\in U_{i}$. For a spanning subgraph $H'\subseteq H$,
define
\[
c_{H,H'}=\sum_{(u_{1},\ldots,u_{v(H)})}p^{e(H)-e(H')}\prod_{(v_{i},v_{j})\in E(H')}1_{G}(u_{i},u_{j}),
\]
the sum obtained by replacing $1_{G}(u_{i},u_{j})$ by $p$ for all
the edges $(v_{i},v_{j})$ of $H$ not in $H'$. Let $H_{0}\subset H_{1}\subset\cdots\subset H_{e(H)} = H$
be a maximal filtration of $H$ by spanning subgraphs, so that for
each $1\le k\le e(H)$, $H_{k}$ has exactly one more edge than $H_{k-1}$.
Let $e_{k}=(v_{i_{k}},v_{j_{k}})$ be the edge introduced in $H_{k}$.
We will show that for all $1\le k\le e(H)$,
\begin{equation}
|c_{H,H_{k}}-c_{H,H_{k-1}}|\le4\delta v(G)^{v(H)}.\label{eq:counting-induction}
\end{equation}

We also know that
\[
c_{H,H_{0}}=\sum_{(u_{1},\ldots,u_{v(H)})}p^{e(H)}=p^{e(H)}\prod_{i=1}^{v(H)}|U_{i}|,
\]
so since $c_{H,H_{e(H)}}=c_{H}$, summing inequality \eqref{eq:counting-induction}
over $k$ and applying the triangle inequality would complete the
proof of \eqref{eq:counting general}. 

Notice that
\begin{equation}
c_{H,H_{k}}-c_{H,H_{k-1}}=\sum_{(u_{1},\ldots,u_{v(H)})}p^{e(H)-k}(1_{G}(u_{i_{k}},u_{j_{k}})-p)\prod_{(v_{i},v_{j})\in E(H_{k-1})}1_{G}(u_{i},u_{j}).\label{eq:counting-difference}
\end{equation}
In the product on the right, all factors depend on at most one of
$u_{i_{k}}$ and $u_{j_{k}}$. Let $T_{k}=\prod_{i\not\in\{i_{k},j_{k}\}}U_{i}$
be the set of all $(v(H)-2)$-tuples of choices of all $u_{i}$ except
for $u_{i_{k}}$ and $u_{j_{k}}$. Write $1_{U}(\cdot)$ for the indicator
function of a vertex set $U$. For each tuple $t\in T_{k}$, we may
define $\{0,1\}$-valued functions
\begin{eqnarray*}
f_{t}(u) & = & 1_{U_{i_{k}}}(u)\prod_{(v_{i_{k}},v_{j})\in E(H_{k-1})}1_{G}(u,u_{j})\\
g_{t}(u) & = & 1_{U_{j_{k}}}(u)\prod_{(v_{i},v_{j_{k}})\in E(H_{k-1})}1_{G}(u_{i},u)
\end{eqnarray*}
which depend on the choice of $t\in T_{k}$ via the choices
of $u_{i}$. Define $H_{k}^{*}$ to be the graph obtained by removing
the two vertices $u_{i_{k}}$ and $u_{j_{k}}$ from $H_{k}$. Then,
we can separate out the factors in the sum in \eqref{eq:counting-difference}
that depend on $u_{i_{k}},u_{j_{k}}$ to find
\[
c_{H,H_{k}}-c_{H,H_{k-1}}=p^{e(H)-k}\sum_{t\in T_{k}}\prod_{(v_{i},v_{j})\in H_{k}^{*}}1_{G}(u_{i},u_{j})\sum_{u_{i_{k}},v_{i_{k}}\in V(G)}(1_{G}(u_{i_{k}},u_{j_{k}})-p)f_{t}(u_{i_{k}})g_{t}(u_{j_{k}}).
\]

The absolute value of the inner double
sum is bounded by $4\delta v(G)^{2}$ because of inequality \eqref{eq:regularity1},
and the product term takes values in $\{0,1\}$. Thus,
\[
|c_{H,H_{k}}-c_{H,H_{k-1}}|\le p^{e(H)-k}v(G)^{v(H)-2}\cdot4\delta v(G)^{2}\le4\delta v(G)^{v(H)},
\]
which proves \eqref{eq:counting-induction}. It follows that
\[
c_{H}=c_{H,H_{e(H)}}=c_{H,H_{0}}+\sum_{k=1}^{e(H)}(c_{H,H_{k}}-c_{H,H_{k-1}})=p^{e(H)}\prod_{i=1}^{v(H)}|U_{i}|\pm4e(H)\delta v(G)^{v(H)},
\]
as desired.
\end{proof}
Finally, we require a lemma of Erd\H os, Goldberg, Pach, and Spencer
\cite{ErdosGoldbergPach} on discrepancy.
\begin{lemma}
\label{lem:discrepancy}Let $G$ be a graph with edge density $q$.
If there is a subset $S\subseteq V(G)$ for which $|e(S)-q\binom{|S|}{2}|\ge D$,
then there exists a set $S'\subseteq V(G)$ of order $\frac{1}{2}v(G)$
such that 
\[
\Big|e(S')-q\binom{|S'|}{2}\Big|\ge(\frac{1}{4}-o(1))D,
\]
where $o(1)$ goes to zero as $D\rightarrow\infty$.
\end{lemma}
In fact, we will only need the following corollary.
\begin{corollary}
\label{cor:discrepancy}Let $G$ be a graph with edge density $q$.
If there is a subset $S\subseteq V(G)$ for which $|e(S)-q\binom{|S|}{2}|\ge D$,
then there exist two disjoint subsets $X,Y\subseteq V(G)$ of size
$\frac{1}{4}v(G)$ such that
\[
|e(X)-e(Y)|\ge\frac{1}{16}D-o(v(G)^{2}),
\]
where the error term is a function of $v(G)\rightarrow\infty$.
\end{corollary}
\begin{proof}
Apply Lemma~\ref{lem:discrepancy} to find a set $S'$ with the stated
properties. Now, pick a uniformly random subset $A\subseteq V(G)$,
obtained by independently picking each vertex of $G$ with probability
$1/2$. Let $X=A\cap S'$, and pick a uniformly random subset $Y\subseteq V(G)\backslash A$.
We now prove that the sizes of $|X|,|Y|,e(X),$ and $e(Y)$ are tightly
concentrated about their expected values.

The marginal distribution of $X$ is the uniform distribution on subsets
of $S'$. Thus, we expect $|X|$ to be tightly concentrated about
$\frac{1}{2}|S'|=\frac{1}{4}v(G)$ and $e(X)$ to be tightly concentrated
about $\frac{1}{4}e(S')$. Meanwhile, the marginal distribution of
$Y$ is the distribution obtained by independently adding each $v\in V(G)$
to $Y$ with probability $\frac{1}{4}$, so we expect $|Y|$ to be
tightly concentrated about $\frac{1}{4}v(G)$ and $e(Y)$ to be tightly
concentrated about $\frac{1}{16}e(G)$.

We first check the tight concentrations of orders. Note that $|X|\sim B(\frac{1}{2}v(G),\frac{1}{2})$
and $|Y|\sim B(v(G),\frac{1}{4})$ are both binomial distributions
with the same mean. By a standard application of Chernoff bounds,
we have that for any $\delta\in[0,1]$,
\[
\Pr\Big[\Big||X|-\frac{1}{4}v(G)\Big|\le\frac{1}{4}\delta v(G)\Big]\le e^{-\delta^{2}v(G)/12},
\]
and the same bound holds for $|Y|$. In particular, w.h.p. $|X|=|Y|=(\frac{1}{4}+o(1))v(G)$.

Next, let $A_{i}$ be the event that the $i$-th vertex of $S'$ lies
in $A$, and define
\[
E_{i}=\mathbb{E}[e(X)|A_{1},\ldots,A_{i}]
\]
to be the expected number of edges in $S'$, conditioned on the first
$i$ events of the $A_{i}$. Notice that $E_{|S'|}=e(X)$ and $E_{0},\ldots,E_{|S'|}$
is a martingale satisfying the $v(G)$-Lipschitz condition, since
adding or removing a single vertex from $X$ changes the number of
edges by at most $v(G)$. The expected value of $e(X)$ is $\frac{1}{4}e(S')$
by linearity of expectation, so the Azuma-Hoeffding inequality gives
\[
\Pr[|e(X)-\frac{1}{4}e(S')|\ge t]\le2\exp\Big(-\frac{t^{2}}{v(G)^{3}}\Big).
\]

Thus, w.h.p. $e(X)=\frac{1}{4}e(S')+o(v(G)^{2})$. By the same argument,
w.h.p. $e(Y)=\frac{1}{16}e(G)+o(v(G)^{2})$.

Since all four concentration events happen with high probability,
there exists some two disjoint sets $X\subseteq S'$ and $Y\subseteq V(G)$
for which $|X|=|Y|=(\frac{1}{4}+o(1))v(G)$, $e(X)=\frac{1}{4}e(S')+o(v(G)^{2})$,
and $e(Y)=\frac{1}{16}e(G)+o(v(G)^{2})$. Add or remove $o(v(G))$
vertices (and therefore $o(v(G)^{2})$ edges) to $X$ and $Y$, we
obtain two sets of the desired order $\frac{1}{4}v(G)$. Their edge
counts still satisfy
\[
|e(X)-e(Y)|=|\frac{1}{4}e(S')-\frac{1}{16}e(G)|+o(v(G)^{2})=\frac{1}{16}D+o(D)+o(v(G)^{2})
\]
as desired.
\end{proof}

\section{The main lemma}

With the stronger quasirandomness condition $\mathcal{R}'_{H,p}(\delta)$,
we are ready to prove the main lemma, which is motivated by an argument
of Conlon, Fox, Sudakov~\cite{ConlonFoxSudakov}. Fix a vertex $v_{0}\in H$
with degree $r$. We will want to control the average difference of
$|d^{r}(u)-d^{r}(v)|$ over all pairs $u,v\in G$, where $d^{r}(u)$
is the $r$-th power of the degree of $u$.
\begin{definition}
Let $H'=H\backslash\{v_{1}\}$. Define $c(u,v)$ to be the number
of pairs of homomorphisms $\phi,\psi:H\rightarrow G$ such that $\phi(v_{1})=u$,
$\psi(v_{1})=v$, and $\phi|_{H'}=\psi|_{H'}$, i.e. the number of
copies of $H'$ in $G$ that extend to a copy of $H$ when we add
either $u$ or $v$ for the first vertex.
\end{definition}
In the language of Reiher and Schacht~\cite{ReiherSchacht}, $c(u,v)$
counts copies of the graph $K$ obtained from $H$ by doubling the
first vertex and fixing the images of these two doubles to be $u$
and $v$.
\begin{lemma}
\label{lem:main}Let $H$ be a graph with a vertex $v_{0}$ of degree
$r$, and let $G$ be a graph that satisfies $\mathcal{R}_{H,p}(\delta)$.
Then,
\[
\sum_{u,v\in V(G)}|d^{r}(u)-d^{r}(v)|=O(p^{e(H)}\delta v(G)^{r+2}).
\]
\end{lemma}
\begin{proof}

First we give a heuristic for the average value of $c(u,v)$. Counting
pairs of homomorphisms $\phi,\psi$ with the given property is the
same as counting the number of homomorphisms $\psi:H\rightarrow G$
for which $\psi(v_{i})\in N(u)$ whenever $(v_{1},v_{i})\in E(H)$.
If the $r$ such vertices are $v_{2},\ldots,v_{r+1}$ (without loss
of generality), then 
\[
c(u,v)=c(H,G;\{v_{1}\},\overbrace{N(u),\ldots,N(u)}^{r},\overbrace{V(G),\ldots,V(G)}^{v(H)-r-1}).
\]
By this formula, we expect that on average,
\[
c(u,v)\approx p^{e(H)}d^{r}(u)v(G)^{v(H)-r-1}.
\]

Turning this approximation on its head, we expect $d^{r}(u)\approx p^{-e(H)}v(G)^{-v(H)+r+1}c(u,v)$,
and so it is natural to bound
\[
|d^{r}(u)-d^{r}(v)|\le|d^{r}(u)-p^{-e(H)}v(G)^{-v(H)+r+1}c(u,v)|+|d^{r}(v)-p^{-e(H)}v(G)^{-v(H)+r+1}c(u,v)|
\]
via the triangle inequality. Summing over all pairs $(u,v)$, we obtain

\begin{eqnarray}
\sum_{u,v\in V(G)}|d^{r}(u)-d^{r}(v)| \le 2\sum_{u,v\in V(G)}|d^{r}(u)-p^{-e(H)}v(G)^{-v(H)+r+1}c(u,v)|\nonumber \\
 = 2p^{-e(H)}v(G)^{-v(H)+r+1}\sum_{u,v\in V(G)}|p^{e(H)}d^{r}(u)v(G)^{v(H)-r-1}-c(u,v)|.\label{eq:d-difference}
\end{eqnarray}

For a fixed vertex $u$, the relevant sum over $v$ is
\[
\Sigma_{u}=\sum_{v\in V(G)}|c(u,v)-p^{e(H)}d^{r}(u)v(G)^{v(H)-r-1}|.
\]

Break up $V(G)=V^{+}\cup V^{-}$ so that $c(u,v)-p^{-e(H)}d^{r}(u)v(G)^{v(H)-r-1}$
is nonnegative on $v\in V^{+}$ and negative on $v\in V^{-}$, and
let $\Sigma_{u}^{+},\Sigma_{u}^{-}$ be the pieces of the sum $\Sigma_{u}$
supported on $V^{+}$, $V^{-}$ respectively. For the positive side,
we get
\begin{eqnarray*}
\Sigma_{u}^{+} & = & \sum_{v\in V^{+}}\Big(c(u,v)-p^{e(H)}d^{r}(u)v(G)^{v(H)-r-1}\Big)\\
 & = & \sum_{v\in V^{+}}c(u,v)-p^{e(H)}|V^{+}|d^{r}(u)v(G)^{v(H)-r-1}.
\end{eqnarray*}

On the other hand, for a fixed $u$, the sum of $c(u,v)$ over $v\in V^{+}$
is exactly
\[
\sum_{v\in V^{+}}c(u,v)=c(H,G;V^{+},\overbrace{N(u),\ldots,N(u)}^{r},\overbrace{V(G),\ldots,V(G)}^{v(H)-r-1}),
\]
a quantity which we can estimate using the given quasirandomness condition. By $\mathcal{R}_{H,p}(\delta)$ it follows that
\[
|\Sigma_{u}^{+}|\le\delta v(G)^{v(H)},
\]
and the same argument shows that
\[
|\Sigma_{u}^{-}|\le\delta v(G)^{v(H)}.
\]

Returning to \eqref{eq:d-difference}, we get the desired bound by the triangle inequality:
\begin{eqnarray*}
\sum_{u,v}|d^{r}(u)-d^{r}(v)| & \le & 2p^{-e(H)}v(G)^{-v(H)+r+1}\sum_{u}(|\Sigma_{u}^{+}|+|\Sigma_{u}^{-}|)\\
 & = & O(p^{-e(H)}\delta v(G)^{r+2}).
\end{eqnarray*}
\end{proof}

\section{Proof of the main theorem}

Conlon, Fox, and Sudakov~\cite{ConlonFoxSudakov} prove the following
elementary inequality to fully exploit the degree bound obtained in Lemma~\ref{lem:main}.
\begin{lemma}
\label{lem:inequality}(Corollary 2.2 from~\cite{ConlonFoxSudakov}.)
Let $a_{1},\ldots,a_{n}$ and $b_{1},\ldots,b_{n}$ be two sets of
$n$ non-negative integers. Then, for any $r\in\mathbb{N}$,
\[
\sum_{i,j=1}^{n}|b_{j}^{r}-a_{i}^{r}|\ge\sum_{j=1}^{n}b_{j}^{r-1}\cdot\Big(\sum_{j=1}^{n}b_{j}-\sum_{i=1}^{n}a_{i}\Big).
\]
\end{lemma}
Now we finish the proof of Theorem~\ref{thm:main} in the following
form.
\begin{lemma}
\label{lem:finish}If $H$ is a nonempty graph and $G$ is a graph
satisfying $\mathcal{R}'_{H,p}(\delta)$, then $G$ satisfies $\mathcal{P}_{K_{2},p}^{*}(\varepsilon)$
for some $\varepsilon=O_{H}(p^{-2e(H)}\delta)$.
\end{lemma}
\begin{proof}
It is easy to check that removing isolated vertices from $H$ has no effect on the quasirandomness condition $\mathcal{R}_{H,p}^{'}(4e(H)\delta)$. We thus assume that $H$ has no isolated vertices.

We may assume that $\delta=o(p^{2e(H)})$ is small compared to $p^{e(H)}$, or else the desired result would be immediately
true with $\varepsilon=1$. Also, because $G$ satisfies $\mathcal{R}'_{H,p}(\delta)$,
\[
c(H,G)\ge p^{e(H)}v(G)^{v(H)}-\delta v(G)^{v(H)}.
\]
Thus, by Lemma~\ref{lem:maxcopies},
\[
p^{e(H)}v(G)^{v(H)}= O(c(H,G)) = O(e(G)^{v(H)}),
\]
which gives a lower bound
\begin{equation}
q=\Omega(p^{e(H)/v(H)})\label{eq:crude}
\end{equation}
on the edge density $q$ of $G$. Here and henceforth all implicit
constants are allowed to depend only on $H$.

Let $r$ be the minimum degree of $H$. We show that $G$ satisfies
$\mathcal{P}_{K_{2},q}^{*}(\gamma)$ where $\gamma=O(q^{-r+1}p^{-e(H)}\delta)$, 
and then that $|p-q|$ is small. Suppose $G$ does not satisfy
$\mathcal{P}_{K_{2},q}^{*}(\gamma)$ for some $\gamma=Cq^{-r+1}p^{-e(H)}\delta$.
It must therefore contain a subset $S$ such that 
\[
|c(K_{2},G[S])-q|S|^{2}|>\gamma v(G)^{2}.
\]

But $c(K_{2},G[S])=2e(S)$, so it follows that $G$ satisfies the
conditions of Corollary~\ref{cor:discrepancy} with $\delta=\frac{1}{2}\gamma-o(1)$.
Applying Corollary~\ref{cor:discrepancy} with this $\delta$, we
may pick $X,Y\subseteq V(G)$ of size $v(G)/4$ for which
\begin{equation}
|e(X)-e(Y)|\ge\Big(\frac{1}{32}\gamma-o(1)\Big)v(G)^{2}.\label{eq:edge-discrep}
\end{equation}

Without loss of generality, assume $e(X)>e(Y)$. The induced subgraph
$G'=G[X\cup Y]$ satisfies $\mathcal{R}_{H,p}(\delta)$ for the same
$\delta$ as $G$ since this property is hereditary, so applying Lemma
\ref{lem:main} to $G'$ and writing $d_{G'}(u)$ for the degree in
$G'$ of $u$, we find that
\begin{equation}
\sum_{u,v\in V(G')}|d_{G'}^{r}(u)-d_{G'}^{r}(v)|=O(p^{-e(H)}\delta v(G')^{r+2}).\label{eq:upper}
\end{equation}
On the other hand, the differences $d_{G'}^{r}(u)-d_{G'}^{r}(v)$
appear in this sum for every pair $(u,v)\in X\times Y$, and these
are large on average by Lemma~\ref{lem:inequality}:
\begin{eqnarray*}
\sum_{u,v\in V(G')}|d_{G'}^{r}(u)-d_{G'}^{r}(v)| & \ge & \sum_{u\in X}\sum_{v\in Y}|d_{G'}^{r}(u)-d_{G'}^{r}(v)|\\
 & \ge & \sum_{u\in X}d_{G'}^{r-1}(u)\Big(\sum_{u\in X}d_{G'}(u)-\sum_{v\in Y}d_{G'}(v)\Big)\\
 & \ge & |X|\Big(\frac{1}{|X|}\sum_{u\in X}d_{G'}(u)\Big)^{r-1}(e(X)-e(Y)).
\end{eqnarray*}
Now, it is easy to see from the proof of Corollary~\ref{cor:discrepancy}
that we choose $X$ to have more edges than average, so that
\[
\frac{1}{|X|}\sum_{u\in X}d_{G'}(u)\ge\frac{2e(X)}{|X|}\ge(q-o(1))|X|.
\]
Together with \eqref{eq:edge-discrep}, it follows that
\[
\sum_{u,v\in V(G')}|d_{G'}^{r}(u)-d_{G'}^{r}(v)|\ge\Omega(q^{r-1}|X|^{r}\cdot\gamma v(G)^{2})\ge\Omega(q^{r-1}\gamma v(G)^{r+2}),
\]
since $|X|=\frac{1}{4}v(G)$. Together with the upper bound \eqref{eq:upper},
we have shown
\[
q^{r-1}\gamma=O(p^{-e(H)}\delta),
\]
which contradicts our choice of $\gamma=Cq^{-r+1}p^{-e(H)}\delta$
if the constant $C$ is sufficiently large. Therefore, $G$ must satisfy
$\mathcal{P}_{K_{2},q}^{*}(\gamma)$ for some $\gamma=O(q^{-r+1}p^{-e(H)}\delta)$.
But then Lemma~\ref{lem:counting} tells us that $G$ also satisfies
$\mathcal{R}'_{H,q}(4e(H)\gamma)$. 

Since $G$ is to satisfy both
$\mathcal{R}'_{H,p}(\delta)$ and $\mathcal{R}'_{H,q}(4e(H)\gamma)$,
we have simultaneously that 
\[
c(H,G)=p^{e(H)}v(G)^{e(H)}\pm\delta v(G)^{e(H)},
\]
and 
\[
c(H,G)=q^{e(H)}v(G)^{e(H)}\pm4e(H)\gamma v(G)^{e(H)}.
\]

By the
triangle inequality,
\begin{eqnarray*}
|p^{e(H)}-q^{e(H)}| & \le & 4e(H)\gamma+\delta,\\
 & = & O(q^{-r+1}p^{-e(H)}\delta).
\end{eqnarray*}

To compare $q$ to $p$, we have the lower bound (\ref{eq:crude}). Recalling that $r$ is the
degree of a single vertex of $H$, we find that $r\le v(H)$, so $q^{-r+1}=O(p^{-e(H)})$. As a result, 
\[
|p^{e(H)}-q^{e(H)}|\le O(p^{-2e(H)}\delta).
\]

Since $G$ satisfies $\mathcal{P}_{K_{2},q}^{*}(\gamma)$, it must also satisfy $\mathcal{P}_{K_{2},p}^{*}(\varepsilon)$ for 
\[
\varepsilon=\gamma+|p^{e(H)}-q^{e(H)}|=O(p^{-2e(H)}\delta),
\]
as desired.
\end{proof}
Combining Lemma~\ref{lem:disjoint} with Lemma~\ref{lem:finish},
Theorem~\ref{thm:main} is proved.

\section*{Acknowledgements}

The author would like to thank Jacob Fox for calling attention to
this problem and for many helpful discussions, and David Conlon and
Mathias Schacht for stimulating conversations. The referee's comments were
also very useful.

\end{document}